\documentclass{birkjour}

\usepackage{amsmath,amssymb}
\usepackage{enumerate}

 \newtheorem{thm}{Theorem}[section]
 \newtheorem{corollary}[thm]{Corollary}
 \newtheorem{lemma}[thm]{Lemma}
 \newtheorem{proposition}[thm]{Proposition}
 \theoremstyle{definition}
 
 \theoremstyle{remark}

 \numberwithin{equation}{section}

\begin{document}

\title[Gromov hyperbolicity of the Kobayashi metric]{Gromov hyperbolicity of the Kobayashi \newline metric on $\mathbb{C}$-convex domains}

\author{Nikolai Nikolov}
\address{Institute of Mathematics and Informatics\\Bulgarian Academy
of Sciences\\ Acad. G. Bonchev 8, 1113 Sofia, Bulgaria
\vspace{0.1cm}
\newline Faculty of Information Sciences\\
State University of Library Studies and Information Technologies\\
Shipchenski prohod 69A, 1574 Sofia,
Bulgaria}

\email{nik@math.bas.bg}

\author{Maria Trybu\l a}

\address{Faculty of Mathematics and Informatics\\Adam Mickiewicz University\\ Umultowska 87, 61-614 Pozna\'{n}, Poland}

\email{maria.h.trybula@gmail.com}

\thanks{The first named author was partially supported by the Bulgarian National Science Fund,
Ministry of Education and Science of Bulgaria under contract DN 12/2.
The second named author was supported by National Center of Science (Poland), grant no. 2013/10/A/ST1/00091.}



\begin{abstract}
In this paper we study the global geometry of the Kobayashi metric on $\mathbb{C}$-convex domains. We provide new examples of non-Gromov hyperbolic domains in $\mathbb{C}^n$ of many kinds: pseudoconvex and non-pseudo
\newline -convex, bounded and unbounded.
\end{abstract}

\maketitle
\section{Introduction}

The Kobayashi distance, denoted $k_\Omega,$ is a pseudodistance associated to any domain $\Omega,$ which has a number of important properties, for instance, the distance decreasing property.
It is known that for bounded domains the Kobayashi pseudodistance is actually a distance. Further, the distance non-increasing property of the  Kobayashi distance is useful for studying holomorphic maps and the geometry of the Kobayashi distance has played an important role in the proofs of many results in several complex variables (see for instance \cite{Fornaess}).

Generally the Kobayashi metric is not Riemannian, and thus no longer has a curvature. Instead one can consider a coarser notion of non-positive curvature from geometric group theory that has the origins in the fundamental work of Alexandrov (see a survey article \cite{Alexandrov}).
Alexandrov gave several equivalent definitions of what it means for a metric space to have curvature bounded from above by a real number $\kappa.$ It was observed by M. Gromov that one of them, given by the co-called CAT($\kappa$) inequality, allows one to faithfully reflect the same concept in a much wider setting - that of geodesic metric spaces (see \cite{Gromov}). Motived by this Gromov introduced the notion of {\it $\delta$-hyperbolicity}, also known as {\it Gromov hyperbolicity}.

Gromov hyperbolic metric spaces have a number of important properties. We shall mention here only a few of them. For instance, it is known that in a Gromov hyperbolic space every quasi-geodesic is within a bounded distance of an actual geodesic (cf. \cite[Chapter III.H, Theorem 1.7]{Bridson}). This can be very helpful in many situations because it is often easy to construct quasi-geodesics but difficult to find actual geodesics. Moreover, every quasi-isometry $f:\Omega_1\rightarrow \Omega_2$ between Gromov hyperbolic spaces have a continuous extension to natural compactifications of $\Omega_1$ and $\Omega_2$ (cf. \cite[Chapter III.H, Theorem, 3.9]{Bridson}).

In \cite{NTT}, Thomas and the authors started considering the sensitivity of Gromov hyperbolicity to removing some part from a domain. Roughly speaking, they observed that under certain conditions on $\Omega_2,$ the space $(\Omega_1\setminus\overline{\Omega}_2,k_{\Omega_1\setminus\overline{\Omega}_2})$ is hyperbolic if and only if  $(\Omega_1,k_{\Omega_1})$ is hyperbolic.
\begin{proposition}\textup{(\cite[Proposition 6]{NTT})}\label{Pascal}
Let $\Omega$ be a bounded domain in $\mathbb{C}^n,\,n\geq 2.$ Assume that $K$ is a linearly convex compact subset of $\Omega.$
Then $\Omega\setminus K$ is a domain such that $k_{\Omega\setminus K}$ is
quasi-isometrically equivalent to $k_{\Omega}|_{(\Omega\setminus K)\times(\Omega\setminus K)}.$

In particular, if $(\Omega,k_\Omega)$ is Gromov hyperbolic, then so is $(\Omega\setminus K,k_{\Omega\setminus K}).$
\end{proposition}
The situation when we cut out a non-compact subset is completely different.
\begin{thm}\textup{(\cite[Theorem 3.1]{paris})}\label{paris them}
Let $\Omega$ be a bounded convex open set in $\mathbb{C}^n$ and $S$ be a complex affine hyperplane such that $\Omega\cap S$ is not empty. Then $(\Omega\setminus S,k_{\Omega\setminus S})$ is not Gromov hyperbolic.
\end{thm}
On the other hand, if $A$ is relatively closed in $\Omega,$ and $A$ is negligible with respect to the $(2n-2)$-dimensional Hausdorff measure, then
$$k_{\Omega\setminus A}=k_{\Omega}|_{(\Omega\setminus A)\times (\Omega\setminus A)}$$
(cf. \cite[Theorem 3.4.2]{Jarnicki}).
Hence, the case considered by Haggui and Chrih is the essential one.
One of our goal is to refine the construction used in \cite{paris}, and show that what is really needed in Theorem \ref{paris them} is a $\mathbb{C}$-convexity.
\begin{thm}\label{paris we}
Let $\Omega\subset\mathbb{C}^n$ be a bounded $\mathbb{C}$-convex domain and $S$ a complex affine hyperplane such that $\Omega\cap S$ is not empty. Then $(\Omega\setminus S,k_{\Omega\setminus S})$ is not Gromov hyperbolic.
\end{thm}
Observe that $\Omega\setminus S$ is generally not $\mathbb{C}$-convex, but it is still very close to $\mathbb{C}$-convex domains. We will show that a slight modification of the proof of Theorem \ref{paris we} allows to carry on the cutting procedure, and the resulting domain is not Gromov hyperbolic as well.

\begin{corollary}\label{finite weakly}
Let $\Omega\subset\mathbb{C}^n$ be a bounded $\mathbb{C}$-convex domain in $\mathbb{C}^n$ and $\mathcal{S}\not=\varnothing$ be a finite family of complex affine hyperplanes intersecting $\Omega.$ Then $(\Omega\setminus\,\bigcup\mathcal{S},k_{\Omega\setminus\,\bigcup\mathcal{S}})$ is not Gromov hyperbolic.
\end{corollary}

Corollary \ref{finite weakly} applies to convex domains also, which is not so obvious if one wants to deduce it directly from Theorem \ref{paris them}.

\begin{corollary}
Let $\Omega\subset\mathbb{C}^n$ be a bounded convex domain and $\mathcal{S}\not=\varnothing$ a finite family of complex affine hyperplanes intersecting $\Omega.$ Then $(\Omega\setminus\,\bigcup\mathcal{S},k_{\Omega\setminus\,\bigcup\mathcal{S}})$ is not Gromov hyperbolic.
\end{corollary}

The key ingredient in the proof of Theorem \ref{paris we} is that every weakly linearly convex domain (and so $\mathbb{C}$-convex) contains a "sufficiently" big convex set. This simple observation allows us to apply the construction of quasi-triangles considered in \cite{paris}. However, Theorem \ref{paris we} does not follow directly from \cite{paris}.

On the other hand, in view of the fact that there are the Riemann Singularity Removable Theorems for bounded holomorphic functions and square integrable holomorphic functions, the phenomena described above do not hold for the Bergman and the Carath\'{e}odory distances.

\begin{proposition}
Let $\Omega\subset\mathbb{C}^n$ be a bounded domain and $S$ a complex affine hyperplane such that $\Omega\cap S\not=\varnothing.$ Then
$(\Omega,d_\Omega)$ is Gromov hyperbolic if and only if $(\Omega\setminus S,d_{\Omega\setminus S})$ is Gromov hyperbolic, where $d_{\Omega}$ denotes the Bergman or the Carath\'{e}odory distance on $\Omega.$
\end{proposition}

It is natural to ask if Theorem \ref{paris them} or \ref{paris we} can be localized. The situation when $S'\subset S\cap \Omega$ touches $\partial\Omega$ essentially differs from the original one since there is not pseudoconvexity or $k$-completeness. Nonetheless, under additional conditions we are able to obtain a partial result.

\begin{thm}\label{thm strictly}
Let $\Omega\subset\mathbb{C}^n$ be a convex domain containing no complex line. Suppose there is a closed subset $S'$ of $\Omega$ such that there are a complex affine line $S,$ a point $\zeta\in\partial\Omega\cap S$ and $R>0$ so that $$\Omega\cap\mathbb{B}(\zeta,R)\cap S'=\Omega\cap\mathbb{B}(\zeta,R)\cap S$$
and $\partial\Omega$ is strongly convex near $\zeta,$ i.e., $\partial\Omega$ contains no segment in some neighbourhood of $\zeta.$ Then $(\Omega\setminus S',k_{\Omega\setminus S'})$ is not Gromov hyperbolic.
\end{thm}

In the proof of Theorem \ref{thm strictly} we observe an interesting  pheonemana, namely that the conglomerate of quasi-geodesics gives again a quasi-geodesic (!).  The proof is purely geometric, relies undoubtedly on the local strict convextity, and does not use the boundness of $\Omega$  (only its $k$-hyperbolicity). It would be interesting to know whether the condition of the local strict convexity is superfluous and could be removed.

In \cite{paris} Haggui and Chrih also investigated some Hartogs type domains. In the present paper we easily obtain the following generalization of \cite[Theorem 3.2]{paris}:

\begin{thm}\label{last thm}
Let $\Omega\subset\mathbb{C}^n$ be a bounded $\mathbb{C}$-convex domain and $\varphi$ a bounded function on $\Omega$ so that the set $\{(z,w)\in\Omega\times\mathbb{C}:\,\lVert w\rVert<e^{-\varphi(z)}\}$ is open \textup{(}for instance $\varphi$ is upper semiconituous\textup{)}. Then the domain
$$\Omega_\varphi=\big{\{}(z,w):z\in\Omega,\,0<\lVert w\rVert<e^{-\varphi(z)}\big{\}}$$
endowed with the Kobayashi distance is not Gromov hyperbolic.
\end{thm}
In Theorem \ref{last thm} we only assumed $\varphi$ is just bounded on $\overline{\Omega},$ no assumption about its continuity, or (strict-)plurisubharmonicity (!).

The rest of the paper is organized as follows. In $\S$2 we collect definitions and basic facts that we shall need. $\S$3 is supplementary for the last three sections, where we present proofs of the main theorems.
\section{Preliminaries}
\subsection{Basic notation}
\begin{itemize}
\item For $X\subset\mathbb{C}$ let $X_*=X\setminus\{0\}.$
\item For $r>0$ let $\mathbb{D}(r)=\{z\in\mathbb{C}:|z|<r\},\,\mathbb{D}(1)=\mathbb{D}.$
\item For $z\in\mathbb{C}^n$ let $\lVert z\rVert$ denote the standard Euclidean norm of $z.$
\item For $\zeta \in\mathbb{C}^n,\,r>0$ let $\mathbb{B}(\zeta,r)=\{z\in\mathbb{C}^n:\lVert z-\zeta\rVert<r\},\ \mathbb{B}(0,1)=\mathbb{B}_n.$
\item Given an open set $\Omega\subset\mathbb{C}^n,\,z\in\Omega,\,X\in\mathbb{C}^n$ let
$$\textup{dist}_\Omega (z)=\inf\,\{\lVert z-w\rVert:\,w\in\partial\Omega\},$$
$$\textup{dist}_\Omega (z;X)=\inf\,\{\lVert z-w\rVert:\,w\in(z+\mathbb{C}X)\cap\partial\Omega\}.$$
\end{itemize}

\subsection{Complex convexity}

A domain is said to be:
\begin{itemize}
\item  $\mathbb{C}${\it -convex} if any non-empty intersection with a complex line is a simply connected domain.
\item {\it linearly convex} (respectively: {\it weakly linearly convex}) if each point in its complement (respectively: boundary) belongs to a complex hyperplane disjoint from the domain.
\end{itemize}
The following implications hold:
$$ \mathbb{C}\textup{-convexity}\Rightarrow\textup{linearly convexity}  \\
\Rightarrow\textup{weak linearly convexity.}$$
Let us note that all three notions of complex convexity are different. However, for bounded domains with $\mathcal{C}^1$-smooth boundaries they coincide. In the general case their place is between convexity and pseudoconvexity.

All proofs in this paper use geometric properties of weakly linearly convex domains, therefore we demonstrate the following lemmas:
\begin{lemma}\label{do wniosku}
Suppose that a weakly linearly convex domain $G\subset\mathbb{C}^n$ contains the origin and balanced sets $G_1\ldots,G_d,\, G_j\subset\{0\}^{n_1+\ldots+n_{j-1}}\times\mathbb{C}^{n_j}\times\{0\}^{n-(n_{j}+\ldots+n_d)},\,n_j\in\mathbb{N}_*,\,n_1+\ldots+n_d=n.$ Then $G$ contains the convex hull of $G_1,\,\ldots,G_d.$
\end{lemma}
\begin{proof}
The proof is similar to the one in \cite{ZZ}. For the sake of completeness we will include it. For every $\epsilon\in (0,1)$ there exists a $\delta>0$ so that
$$G_\epsilon=\bigcup_{j=1}^d\big{(}\delta G_1\times\ldots\times\delta G_{j-1}\times\epsilon G_j\times\delta G_{j+1}\times\ldots\times\delta G_{d}\big{)}.$$
Note that $\widehat{G}_\epsilon\subset G,$ where $\widehat{G}_\epsilon$ is the smallest linearly convex domain that contains $G_\epsilon.$ But
$$G_\epsilon=\big{\{}z\in\mathbb{C}^n\big{|}\forall b\in\mathbb{C}^n: \langle z,b\rangle=1\ \exists a\in G_\epsilon: \langle a,b\rangle=1\big{\}}.$$
(see for instance \cite[Proposition 4.6.2]{Hormander}).Then $\widehat{G}_\epsilon$ is not only balanced but also convex (since it is linearly convex). Consequently,
$$E_\epsilon=\big{\{}z\in\mathbb{C}^n:\sum_{j=1}^d\mu_j(z_{n_1+\ldots n_{j}+1},\ldots,z_{n_1+\ldots n_{j+1}})<\epsilon\big{\}}\subset\widehat{G}_\epsilon\subset G.$$
where $\mu_j$ denotes the Minkowski functional corresponding to $G_j.$ Letting $\epsilon\rightarrow 1,$ we get the lemma.
\end{proof}

\begin{lemma}\label{product}
Suppose $\Omega_j\subset\mathbb{C}^{n_j},\,j=1,2,\ n_j\in\mathbb{N}_*,$ are bounded, weakly linearly convex domains. Then $\Omega_1\times\Omega_2$ is weakly linearly convex.
\end{lemma}
Since the proof is short we include it.
\begin{proof}
Fix $\zeta=(\zeta_1,\zeta_2)\in\partial (\Omega_1\times\Omega_2).$ Without loss of generality we may assume that $\zeta_1\in\partial\Omega_1.$ Since $\Omega_1$ is weakly linearly convex, there exists a complex hyperplane $L\subset\mathbb{C}^{n_1}$ so that $(\zeta_1+L)\cap\Omega_1=\varnothing.$ Hence, $$\big{(}\zeta+(L\times\mathbb{C}^{n_2})\big{)}\cap(\Omega_1\times\Omega_2)=\varnothing.$$
This shows that $\Omega_1\times\Omega_2$ is weakly linearly convex.
\end{proof}

More properties of complex convex domains can be found in \cite{Sigurdsson}, \cite{Hormander}.

\subsection{The Kobayashi metric and distance}
Given a domain $\Omega\subset\mathbb{C}^n$ the {\it Kobayashi pseudometric} is defined as follows
\begin{multline*}
\kappa_{\Omega}(z;X)=\inf\big{\{}|\lambda|:\,\textup{there exists }f\in H(\mathbb{D},\Omega) \textup{ so that }  \\
f(0)=z,\,\lambda f'(0)=X\big{\}}.
\end{multline*}
If $\alpha:[a,b]\rightarrow\Omega$ is a $\mathcal{C}^1$ piecewise curve we can define the {\it length of }$\alpha$ to be
$$l_{\kappa_\Omega}(\alpha)=\int_a^b\kappa_\Omega \big{(}\alpha(t);\alpha'(t)\big{)}dt.$$
One can then define the {\it Kobayashi pseudodistance} to be
\begin{multline*}
k_{\Omega}(p,q)=\inf\,\big{\{l_{\kappa_\Omega}}(\alpha):\,\alpha:[0,1]\rightarrow\Omega \textup{ is a piecewise }\mathcal{C}^1\textup{ smooth} \\
\textup{with }\alpha(0)=p,\,\alpha(1)=q\big{\}}.
\end{multline*}
An important property of the Kobayashi pseudodistance is the {\it holomorphic contractibility:} if $f:\Omega_1\rightarrow\Omega_2$ is holomorphic, then
$$k_{\Omega_2}(f(z),f(w))\leq k_{\Omega_1}(z,w).$$
$\Omega$ is $k$-{\it hyperbolic} if $k_\Omega$ is a distance. It easily follows from the distance decreasing property that the Kobayashi pseudodistance for bounded domains is a distance. Further, $k$-hyperbolicity of convex domains is well understood. That is:

\begin{thm}\textup{(\cite{Barth})}
Let $\Omega$ be a geometrically convex domain in $\mathbb{C}^n$ containing no complex affine line. Then the Carath\'{e}odory pseudodistance is a distance, and every closed ball with respect to it is compact. In particular, $\Omega$ is $k$-hyperbolic.
\end{thm}
We say $(\Omega,k_\Omega)$ is $k${\it -finitely compact} if every ball with finite radius is relatively compact in $\Omega$

For further information on the Kobayashi metric/distance we refer the reader to \cite{Jarnicki}, \cite{Kobayashi}.

\subsection{Gromov hyperbolicity}
Let $(X,d)$ be a geodesic metric space.

If $[a,b]\subset\mathbb{R}$ is an interval, a curve $\alpha:[a,b]\rightarrow X$ is a {\it geodesic} if $d(\alpha(s),\alpha(t))=|s-t|$ for all $s,\,t\in [a,b].$ A {\it geodesic triangle} in the metric space is a choice of three points in $X$ and geodesic segments connecting these points. A geodesic triangle is said to be $M$-{\it thin} if any point on any side of the triangle is within the distance of $M$ from the other two sides. A geodesic metric space is said to be ({\it Gromov}) {\it hyperbolic} if there exists $M>0$ so that every geodesic triangle is $M$-thin.

By definition, an $(A,B)$-{\it quasi-geodesic} $\beta:[c,d]\rightarrow X$ satisfies the condition
$$A^{-1}|s-t|-B\leq d(\beta(s),\beta(t)) \leq A|s-t|+B$$
for all $s,\,t\in[c,d].$ An $(A,B)$-{\it quasi-geodesic triangle} consists of three $(A,B)$-quasi-geodesics (its sides). Such a triangle is said to be $M$-{\it thin} if each side lies in the $M$-neighbourhood of the union of the other two sides.

\begin{proposition}\textup{(\cite{Bridson})}\label{prop Gromov}
If $(X,d)$ is hyperbolic, then every $(A,B)$-quasi-geodesic triangle is $M$-thin for some $M>0$ depending on $A,\,B.$
\end{proposition}
The book \cite{Bridson} by Bridson and Haefliger is one of the standard references for Gromov hyperbolic metric spaces.

\section{A Lower Bound For The Kobayashi Distance}

\begin{lemma}\label{punctured disc}
For any $z,\,w\in\mathbb{D}_*$ we have
$$k_{\mathbb{D}_*}(z,w)\geq\frac{1}{2}\Big{|}\log\frac{\log |w|}{\log |z|}\Big{|}.$$
The equality holds if and only if $\frac{w}{z}\in\mathbb{R}_{>0}.$
\end{lemma}
\begin{proof}
Using a rotation we may assume $z\in(0,1)$ and $w=|w|e^{i\varphi}$ for some $\varphi\in [-\pi,\pi).$ Now recall that the map $$\pi:\mathbb{H}_+=\{\zeta\in\mathbb{C}: \Im \zeta>0\}\rightarrow\mathbb{D}_*$$
$$\pi(z)=e^{iz}$$
is a holomorphic covering. Hence
$$k_{\mathbb{D}_*}(z,w)=\inf\big{\{}k_{\mathbb{H}_+}(-i\log z,-i\log |w|+\varphi+2k\pi )\,:\,k\in\mathbb{Z}\big{\}}$$
(cf. \cite[Theorem 3.3.7]{Jarnicki})
$$=\inf\big{\{}k_{\mathbb{D}}\big{(}\sigma(-i\log z),\sigma(-i\log |w|+\varphi+2k\pi )\big{)}\,,\,k\in\mathbb{Z}\big{\}},$$
where $\sigma(u)=\frac{u+i\log z}{u-i\log z}.$ $\sigma$ maps the line $\{s\in\mathbb{C}:\,\textup{Im}\,s=\log |w|\}$ onto the circle symmetric with respect to the real axis and orthogonal to $\partial\mathbb{D}.$ Consequently, from the formula for the Kobayashi distance on the disc we get:
\begin{multline*}
k_{\mathbb{D}_*}(z,w)=k_{\mathbb{D}}\big{(}\sigma(-i\log z),\sigma(-i\log |w|+\varphi)\big{)} \\
\geq k_{\mathbb{D}}\big{(}\sigma(-i\log z),\sigma(-i\log |w|)\big{)}=\frac{1}{2}\Big{|}\log\frac{\log |w|}{\log z}\Big{|}.
\end{multline*}
\end{proof}

\begin{corollary}The curve $\alpha:[0,\infty)\rightarrow\mathbb{D}_*,\,\alpha(t)=e^{-e^{2u}}$ is $k_{\mathbb{D}_*}$-geodesic.
\end{corollary}
\begin{proof}
The second part of Lemma \ref{punctured disc} implies that $k_{\mathbb{D}_*}(\alpha(t),\alpha(s))=|s-t|.$
\end{proof}

\begin{lemma}\label{podstawowe}
Suppose $\Omega\subset\mathbb{C}^n$ is a bounded weakly linearly convex domain and $p,\,q\in\Omega$ are distinct. Let $L$ be the complex line containing $p$ and $q.$ If $\zeta\in L\cap\partial\Omega,$ then
$$k_\Omega(p,q)\geq\frac{1}{2}\Big{|}\log\frac{\log\,\frac{\lVert p-\zeta \rVert}{d}}{\log\,\frac{\lVert q-\zeta \rVert}{d}}\Big{|},$$
where $d=\textup{diam}\,\Omega.$
\end{lemma}

\begin{proof}
$\Omega$ is weak linearly convex therefore there exists a complex hyperplane $H$ such that $\Omega\cap\big{(}
\zeta+H\big{)}=\varnothing$ (cf. \cite{Sigurdsson}). Using an affine transformation we may assume $\zeta=0.$ Now, consider the projection $\pi:\mathbb{C}^n\rightarrow\mathbb{C}$ onto $L$ in the direction $H.$ Clearly, $\pi(\Omega)\subset\mathbb{D}(d)_*.$ By the holomorphic contractibility:
$$k_\Omega(p,q)\geq k_{\mathbb{D}(d)_*}(\pi(p),\pi(q))=k_{\mathbb{D}_*}\Big{(}\frac{\pi(p)}{d},\frac{\pi(q)}{d}\Big{)}\geq \frac{1}{2}\Big{|}\log\frac{\log\,\frac{\lVert p-\zeta \rVert}{d}}{\log\,\frac{\lVert q-\zeta \rVert}{d}}\Big{|}$$
(in the last inequality we applied Lemma \ref{punctured disc}).
\end{proof}
The estimate in Lemma \ref{podstawowe} is not optimal even in dimension $1.$ However, it is sufficient for demonstrating the $k$-completeness of every bounded weak linearly convex domain, which, to the best of our knowledge, has not been previously observed. Recall that a domain $\Omega\subset\mathbb{C}^n$ is called $k$-complete if and $k_\Omega$-Cauchy sequence $(z_n)_{n\in\mathbb{N}}\subset\Omega$ converges to a point $z_0\in\Omega$ with respect to the natural topology on $\Omega,$ i.e., $\lVert z_n-z_0\rVert\rightarrow 0.$
\begin{proposition}\label{proposition weak linear}
Let $\Omega\subset\mathbb{C}^n$ be a weakly linearly bounded domain. Then $\Omega$ is $k$-finitely compact. In particular, $\Omega$ is $k$-complete.
\end{proposition}
Proposition \ref{proposition weak linear} is not true for the Carath\'{e}odory distance.
\begin{proof}
Fix $\zeta\in\partial\Omega$ and $p\in\Omega.$ Without loss of generality we may assume $\zeta=0.$ By Lemma \ref{podstawowe}
$$k_{\Omega}(p,z)\geq k_{\mathbb{D}(d)_*}(\pi(p),\pi(z)),$$
where $d$ is the diameter of $\Omega$ and $\pi$ is the projection onto $(p-\zeta)\mathbb{C}+\zeta$ in a direction $H$ such that $\zeta+H\cap\Omega=\varnothing.$
Hence, by the completness of the punctured disc, passing with $z$ to $\zeta,$ we conclude that
$$k_{\Omega}(p,z)\rightarrow\infty.$$
Since $p$ was arbitrary chosen, we have proven that Kobayashi balls with finite radii are relatively compact in $\Omega.$ From this easily follows the second part (see also \cite[Theorem 7.3.2]{Jarnicki}).
\end{proof}
In the next proposition we will discuss the existence of quasi-geodesics for weakly linearly convex open sets. Namely, we will provide a sufficient condition for a line segment with one endpoint in the boundary to be parametrized as a quasi-geodesic.

Let $\{e_j\}$ denote the standard coordinate basis of $\mathbb{C}^n.$
\begin{proposition}\label{theorem weak linear}
Suppose a weakly linearly convex domain $\Omega\subset\mathbb{C}^n$ contains a set $\bigcup_{j=1}^n q+\mathbb{D}(r_j)e_j,$ where $r_j>0$ and $\partial \Omega\cap (q+\partial\mathbb{D}(r_1)e_1)\not=\varnothing.$ Suppose further that there are a projection $\pi:\mathbb{C}^n\rightarrow\mathbb{C}^d,\,1\leq d\leq n$ and $\zeta\in\partial\Omega\cap(q+\partial\mathbb{D}(r_1)e_1)$ such that $\pi(\Omega)$ is $\mathbb{C}$-convex and $\pi(\zeta)\in\partial(\pi(\Omega))$. Then for every non-empty set $\omega\Subset\textup{conv}\big{(}\bigcup_{j>1}\,q+\mathbb{D}(r_j)e_j\big{)}\subset\Omega,$ there are $A,\,B>0$ so that
$$\alpha:[0,\infty)\owns u\mapsto p^u:=\zeta+e^{-2u}(p-\zeta)\in\Omega$$ is $(A,B)$-quasi-geodesic in $(\Omega,k_\Omega)$ provided $p\in\omega.$
\end{proposition}

Since weak linear convexity is invariant under projective transformations, it is not important that discs in Proposition \ref{theorem weak linear} are orthogonal.

\begin{proof}
First of all $\alpha(\mathbb{R}_{\geq 0})\subset\Omega$ (see Lemma \ref{do wniosku}). Let us recall that the statement holds for convex domains (cf. \cite[Proposition 4]{paris}). Hence, by the holomorphic contractibility of the Kobayashi distance, we get the upper estimate for $k_\Omega(p^s,p^t).$ It remains to indicate the lower estimate.
For that purpose, let $\beta:[0,1]\rightarrow\Omega$ be a $\mathcal{C}^1$-smooth curve such that $\beta(0)=p^s,\,\beta(1)=p^t.$ By Proposition \ref{oszacowanie Kobayashi}:
$$
l_{\kappa_{\Omega}}(\beta)\geq l_{\kappa_{\pi(\Omega)}}(\pi\circ\beta)\geq k_{\pi(\Omega)}(\pi(p^s),\pi(p^t))\geq A^{-1}|s-t|-B
$$
for some $A,\,B>0$ depending on $\pi.$
\end{proof}

\section{Proofs of Theorem \ref{paris we} and Corollary \ref{finite weakly}}
In this section we use estimates for the Kobayashi metric and the Kobayashi distance to obtain quasi-geodesic triangles.

The following bounds on the Kobayashi metric are well-known:
\begin{lemma}Suppose $\Omega\subset\mathbb{C}^n$ is domain. Then
$$\kappa_{\Omega}(p;v)\leq\frac{\lVert v\rVert}{\textup{dist}_\Omega(p;v)}$$
for $p\in\Omega$ and $v\in\mathbb{C}^n$ non-zero.
\end{lemma}

\begin{proposition}\textup{(\cite[Proposition 1]{NPZ})} Suppose $\Omega\subset\mathbb{C}^n$ is a $\mathbb{C}$-convex domain. Then
$$\frac{\lVert v\rVert}{\textup{4\,dist}_{\Omega}\,(z;v)}\leq \kappa_{\Omega}(z;v)$$
for $z\in\Omega$ and $v\in\mathbb{C}^n$ non-zero.
\end{proposition}

From this Nikolov and Trybu\l a obtained a lower estimate of the Kobayashi distance:
\begin{proposition}\label{oszacowanie Kobayashi}\textup{(\cite[Proposition 2]{NT})}    Suppose $\Omega\subset\mathbb{C}^n$ is an open $\mathbb{C}$-convex set and $p,\,q\in\Omega$ are distinct. Then
$$\frac{1}{4}\log\Big{(}1+\frac{\lVert p-q\rVert}{\textup{min}\,\{\textup{dist}_\Omega\,(p;q-p),\textup{dist}_\Omega\,(q;p-q)\}}\,\Big{)}\leq k_{\Omega}(p,q).$$
\end{proposition}
\begin{proof}[Proof of Theorem \ref{paris we}]
We can assume $S=\{z_n=0\}$ and $0\in\Omega.$ Fix $r\in (0,\infty)$ so that $r\mathbb{B}_n\subset\Omega.$  Next, choose $s\in (0,r).$ Let $p=(0,\ldots,0,s),\,d=\textup{diam}\,\Omega,$ $r_k=\textup{dist}_\Omega(0;e_k),\,k=1,\ldots,n,$ and $\omega=\textup{conv}\big{(}\bigcup_{k=1}^{n}\mathbb{D}(r_k)e_k\big{)}.$
By Lemma \ref{do wniosku}: $\omega\subset\Omega.$

\vspace{0.1cm}
{\it \textbf{Step I} $\Omega\setminus S$ is $k$-complete.}
\begin{proof} Observe that $\Omega\setminus S$ is a weakly linearly convex domain. Hence, the statement follows directly from Proposition \ref{proposition weak linear}.
\end{proof}
In the next three steps we construct quasi-geodesics in $\Omega\setminus S.$

{\it \textbf{Step II} For every $\zeta\in\partial\omega\cap\partial\Omega\cap S$ the segment $[p,\zeta)$ parametrized as follows:
$$\alpha:[0,\infty)\owns u\mapsto \zeta + e^{-2u}(p-\zeta)=p^u$$
is a quasi-geodesic. Moreover, if $K\Subset\omega\setminus S,$ then there are $A,\,B>0$ so that every segment $[q,\zeta)$
is an $(A,B)$-quasi-geodesic for $q\in K.$}
\begin{proof}Choose $u_1,\,u_2\geq 0.$ By the holomorphic contractibility of the Kobayashi distance we have:
\begin{equation}\label{step 2}
k_{\Omega}(p^{u_1},p^{u_2})\leq k_{\Omega\setminus S}(p^{u_1},p^{u_2})\leq k_{\omega\setminus S}(p^{u_1},p^{u_2}).
\end{equation}
We may find a positive $c>0$ so that
\begin{equation}\label{step 2.1}
\textup{dist}_{\omega\setminus S}(p;p-\zeta)\geq c\lVert p-\zeta\rVert.
\end{equation}
Using the fact that the Kobayashi distance is the integrated form of the Kobayashi metric, we have the following:
$$k_{\omega\setminus S}(p^s,p^t)\leq l_{\kappa_{\omega\setminus S}}(\alpha|_{[t,s]})\leq\int_{t}^s\frac{2\lVert p-\zeta\rVert}{c}du=\frac{2\lVert p-\zeta\rVert}{c}|s-t|.$$
On the other hand, Proposition \ref{oszacowanie Kobayashi} applied to $k_{\Omega}$
 gives:
 $$k_{\Omega}(p^u, p^t)\geq \frac{1}{4}\Big{|}\log\frac{\lVert p^s-\zeta^T\lVert}{\lVert p^t-\zeta^T\lVert}\Big{|}=\frac{1}{2}|s-t|.$$
Combining the last two inequalities with (\ref{step 2}) we obtain the first part.
The above argument can also be used for the remaining part because there exists a positive constant $c'$ so that (\ref{step 2.1}) holds for all $q\in K\Subset \omega\setminus S.$
\end{proof}

{\it \textbf{Step III} The real segment $[p,0)$ parametrized as follows:
$$\beta=(\beta_1,\ldots,\beta_n):[0,\infty)\owns u\mapsto e^{1-e^{2u}}p\in [p,0)=q^u$$
is a quasi-geodesic.}
\vspace{0.1cm}

\begin{proof}Clearly $\{0\}^{n-1}\times\mathbb{D}(r)_* \subset\Omega\setminus S\subset\mathbb{C}^{n-1}\times\mathbb{D}(d)_*.$ By Lemma \ref{punctured disc} and the holomorphic contractibility:
\begin{multline*}
\frac{1}{2}\Big{|}\log\frac{e^{2s}+\log\,\frac{d}{es}}{e^{2t}+\log\,\frac{d}{es}}\Big{|}= k_{\mathbb{D}(d)_*}\big{(}\beta_n(s),\beta_n(t)\big{)}\leq k_{\mathbb{C}^{n-1}\times\mathbb{D}(d)_*} \big{(}\beta(s),\beta(t)\big{)} \\ \leq k_{\Omega\setminus S}\big{(}\beta(s),\beta(t)\big{)}\leq k_{\mathbb{D}(r)_*} \big{(}\beta_n(s),\beta_n(t)\big{)}=\frac{1}{2}\Big{|}\log\frac{e^{2s}+\log\,\frac{r}{es}}{e^{2t}+\log\,\frac{r}{es}}\Big{|}.
\end{multline*}
\end{proof}

\textup{Fix a point $\zeta\in\partial\omega\cap\partial\Omega\cap S$ and $T>0.$ Let  $\zeta^T$ be a unique point that lies in the intersection $\partial\Omega\cap p^T+\mathbb{R}_{>0}\zeta.$ Consider the curve
$$\gamma:[0,\infty)\owns u\rightarrow \zeta^T+e^{-2u}(p^T-\zeta^T)=r^u\in\Omega.$$
Clearly, the segments: $[p,\zeta],\ [p^T,\zeta^T]$ intersect at exactly one point, say $\eta^T.$}
\vspace{0.1cm}

{\it \textbf{Step IV}
There exist $A,\,B>0$ so that the curve
$\gamma|_{[0,\gamma^{-1}(\eta^T)]}$
is an $(A,B)$-quasi-geodesic.}
\begin{proof}We skip the proof since it goes along the same lines as the proof in Step 2.
\end{proof}
\textup{And finally:}

{\it \textbf{Step V} $(\Omega\setminus S,k_{\Omega\setminus S})$ is not Gromov hyperbolic.}
\begin{proof}Assume the contrary. We will show that for any $M>0$ there exists $T>0$ so that the $(A,B)$-quasi-geodesic triangle with sides: $[p,p^T],\,[p^T,\eta^T],$ $[\eta^T,p]$ is not $M$-thin.
By Proposition \ref{proposition weak linear} there is $T_0>0$ so that $k_{\Omega\setminus S}(\eta^{T_0},[p,0))>M.$ Using once more Proposition \ref{proposition weak linear}, find $T>T_0$ so that $k_{\Omega\setminus S}(\eta^{T_0},[p^T,\eta^T])>M.$ Consequently, the constructed $(A,B)$-quasi-geodesic triangle is not $M$-thin. But this contradicts Proposition \ref{prop Gromov}.
\end{proof}
The proof has been completed.
\end{proof}
\begin{proof}[Proof of Corollary \ref{finite weakly}]
Without loss of generality assume $\{ z_n =0 \} \in\mathcal{S}.$ Fix $\zeta\in S\cap\partial\Omega.$ Further, using a linear transformation, we may assume that $\mathbb{B}(r)\times\mathbb{D}(r)\subset\Omega\setminus\bigcup_{S'\in\mathcal{S}\setminus\{S\}}S'$ and $\lVert \zeta\rVert=r$ for some $r>0.$
By Lemma \ref{do wniosku}:
$$\tilde{\Omega}:=\big{\{}z=(z',z_n)\in\mathbb{C}^n:\lVert z'\rVert+|z_n|< 1\big{\}}\setminus S\subset\Omega\setminus\bigcup_{S'\in\mathcal{S}}S'\subset\Omega\setminus S.$$
Replacing $\omega$ by $\tilde{\Omega}$ in the previous proof, we may repeat the whole construction of quasi-geodesic triangles, and so obtain the desired statement (see also the last section).
\end{proof}

\section{Proof of Theorem \ref{thm strictly}}
\begin{proposition}\label{completness}
Suppose $\Omega$  is a bounded convex domain and $S$ is an affine hyperplane  so that $S\cap\Omega$ has a non-empty interior. If $S'\subsetneq S\cap\Omega,$ then
$\Omega\setminus S'$ is not $k$-finitely compact. In particular, $\Omega\setminus S'$ is not pseudoconvex.
\end{proposition}
\begin{proof}
Using a linear transformation, we may assume $S=\{\zeta=(\zeta_1,\zeta'):\, \zeta_1=0\},$ $0\in\Omega \cap (S\setminus S').$ Take $r\in(0,1)$ so that $\overline{\mathbb{D}(2r)}^{\,n}\subset\Omega.$ Fix $w\in\partial\mathbb{D}(r)^n$ and $z_0\in\Omega\setminus S'.$ By the triangle inequality: $
k_{\Omega\setminus S'}(z_0,z)\leq k_{\Omega\setminus S'}(z_0,w)+k_{\Omega\setminus S'}(w,z).$ It suffices therefore to estimate the last term. If $z_1\not=0,$ then:
\begin{multline*}
k_{\Omega\setminus S'}(w,z)\leq k_{\Omega\setminus S'}(w,(w_1,0'))+k_{\Omega\setminus S'}((w_1,0'),(z_1,0'))+k_{\Omega\setminus S'}(z_1,0'),z) \\
\leq k_{\mathbb{D}(2r)^{n-1}}(w',0')+k_{\mathbb{D}(2r)}(w_1,z_1)+k_{\mathbb{D}(2r)^{n-1}}(0',z')\leq 2\log 3.
\end{multline*}
It remains to consider the case $z_1=0:$
\begin{multline*}
k_{\Omega\setminus S'}(w,(0,z'))\leq k_{\Omega\setminus S'}(w,(w_1,z'))+k_{\Omega\setminus S'}((w_1,z'),z)\leq  \\
2\log 3+k_{\mathbb{D}(2r)}(w_1,z_1)\leq 3\log 3.
\end{multline*}
 The last part holds due to \cite[Corollary 14.5.2]{Jarnicki}.
\end{proof}

We shall need the following localization of the Kobayashi metric:
\begin{proposition}\textup{(\cite[Proposition 7.2.9]{Jarnicki})}\label{lokalizacja} Suppose that $\Omega$ is $k$-hyperbolic in $\mathbb{C}^n$ and let $U\subset \Omega$ be any subdomain. Then
\begin{multline*}
\kappa_U(z;X)\leq\inf\big{\{}\coth\,k_\Omega(z,w):\,w\in\Omega\setminus U\big{\}}\kappa_\Omega (z;X),\ \ z\in U,\,X\in\mathbb{C}^n.
\end{multline*}
\end{proposition}
We proceed to the main part of this section.
\begin{proof}[Proof of Theorem \ref{thm strictly}]
In the following $A,\,B$ will be constants depending only on $\Omega.$ The actual values of $A$ and $B$ do not matter and may change within the lines.

Assume the contrary, i.e., $(\Omega\setminus S',k_{\Omega\setminus S'})$ is Gromov hyperbolic. Like in the proof of Theorem \ref{paris we} we construct a family of $(A,B)$-quasi-geodesic triangles which are not $M$-thin for any $M>0.$

Using a rotation, we may assume $S=\{z_1=0\}$ and $\zeta$ is a unique point lying on the intersection $\partial\Omega\cap\,\mathbb{R}_{\geq 0}\times\mathbb{C}^{n-1}.$ Taking smaller $R$ we may assume that every point in $\partial \Omega\cap \mathbb{B}(\zeta,R)$ is strongly convex. Fix $0<r\ll R$ (where from now on $0<a\ll b$ means $r/R$ is small). Let:
$$p\in\Omega\cap\partial\mathbb{B}(\zeta,r)\cap\{\textup{Im}\,z_1=0\},$$
and
$$\eta\in\partial\Omega\cap\partial\mathbb{B}(\zeta,r)\cap\{\textup{Im}\,z_1>0\}.$$

\textbf{Case: $S'\subset S$}

Consider the parametrization of $(p,\eta)$ given by
$$\alpha:\ \,\mathbb{R}_{>0}\owns u\rightarrow  \widehat{\eta}^{\,u}=\eta+e^{-2u}(p-\eta)\in\Omega.$$
Since $k_{\Omega\setminus S}\geq k_{\Omega\setminus S'}\geq k_{\Omega},$ it follows directly from Proposition \ref{theorem weak linear}  that $\alpha$ is an $(A,B)$-quasi-geodesic in $(\Omega\setminus S',k_{\Omega\setminus S'}).$

Fix $u>0.$ Put $p^u=p+e^{-2u}(\eta-p).$ Let $\zeta^u,\,\eta^u\in (\eta,\zeta)$ be so that $\textup{Im}\,\zeta^u_1=\textup{Im}\, p_1^u,\,\textup{Im}\,\eta^u_1=\textup{Im}\,\widehat{\eta}^{\,u}_1.$

Convexity and  Proposition \ref{theorem weak linear} imply that the segments $[p^u,\zeta^u]$ and $[p^u,\eta^u]$ can be parametrized to be $(A,B)$-quasi-geodesics.

In order to finish the construction of quasi-geodesic triangles we shall need the following:
\begin{lemma}\label{quasi segment}
For $u\gg 1$ the segment $(\eta,\zeta)$ can be parametrized to be an $(A,B)$-quasi-geodesic.
\end{lemma}

\begin{proof}
Fix $0<s<s'\ll \frac{r}{2}.$ Let $\eta',\,\zeta'\in[\eta,\zeta]$ be so that $\lVert \eta-\eta'\lVert=s,$ $\lVert \zeta-\zeta'\rVert=s.$

The intervals $[\eta',\eta)$ and $[\zeta',\zeta)$ can be parametrized to be $(A,B)$-quasi-geodesics. Let
 $$\alpha_\eta:[0,\infty)\rightarrow [\eta',\eta),$$
$$\alpha_\zeta:[0,\infty)\rightarrow [\zeta',\zeta)$$
 be these parametrizations. Our goal is to show that $\alpha_\eta\cup [\eta',\zeta']\cup\alpha_\zeta$ is a quasi-geodesic.

Let
$$2M=l_{k_{\Omega\setminus S'}}([\eta',\zeta']),$$
and let $\beta:[0,2M]\rightarrow\Omega\setminus S'$ be the length-parametrization of $[\eta',\zeta'].$ Define
$$\gamma:\mathbb{R}\rightarrow (\eta,\zeta)\subset\Omega\setminus S'$$
as follows:
\[ \gamma(t)=
  \begin{cases}
    \alpha_\eta(-t-M)       & \quad \text{if }\ t\leq -M,\\
    \beta(t-M)  & \quad \text{if }\ -M\leq t\leq M,\\
    \alpha_\zeta(t-M)       & \quad \text{if }\ t\geq M.\\
  \end{cases}
\]
It suffices to consider the case when $t_1\ll -M\leq M\ll t_2.$ From above we easily compute:
\begin{multline*}
k_{\Omega\setminus S'}\big{(}\gamma(t_1),\gamma(t_2)\big{)}\leq k_{\Omega\setminus S'}\big{(}\alpha_\eta(-t_1-M),\eta'\big{)}+k_{\Omega\setminus S'}(\eta',\zeta')+ \\
k_{\Omega\setminus S'}\big{(}\zeta',\alpha_\eta(t_2-M)\big{)}\leq A'|t_2-t_1|+2B'++k_{\Omega\setminus S'}(\eta',\zeta').
\end{multline*}
It remains to prove the estimates from below for $k_{\Omega\setminus S'}\big{(}\gamma(t_1),\gamma(t_2)\big{)}.$ Fix a geodesic $a:[0,1]\rightarrow\Omega\setminus S'$ such that $a(0)=\gamma(t_1),\,a(1)=\gamma(t_2).$ Define:
 $$s_1=\inf\big{\{} u:\,a((u,1))\subset\Omega\setminus(S'\cup\mathbb{B}(\eta,s))\big{\}},$$
and
$$s_2=\sup\big{\{} u:\,a((0,u))\subset\Omega\setminus(S'\cup\mathbb{B}(\zeta,s))\big{\}}.$$
Then
$$l_{k_{\Omega\setminus S'}}(a)\geq k_{\Omega\setminus S'}\big{(}\gamma(t_1),a(s_1)\big{)}+ k_{\Omega\setminus S'}\big{(}\gamma(t_2),a(s_2)\big{)}=(*).$$
By the strong convexity we may assume
\begin{equation}\label{epsilon}
\inf\big{\{}\lVert z-w\rVert :\,z\in\partial\mathbb{B}(\zeta,s)\cap\Omega,\ w\in H+\zeta\big{\}}=\varepsilon>0,
\end{equation}
where $H$ is a supporting hyperplane of $\Omega$ at $\zeta.$ Hence, if $H_\Omega$ denotes the halfspace which contains $\Omega$ and $\partial H_\Omega=H+\zeta,$ then
\begin{multline*}
k_{\Omega\setminus S'}(\gamma (t_2),a(s_2))  \\
\geq \inf\{ k_{H_\Omega}(\gamma(t_2),z):z\in\partial\mathbb{B}(\eta,s)\cap\Omega\}\geq-\frac{1}{2}\log\frac{\textup{dist}_{H_\Omega}(\gamma(t_2))}{\varepsilon}.
\end{multline*}
By the strong convexity of $\Omega$ there is a constant $c>0$ so that
$$\textup{dist}_{H_\Omega}(\gamma(t))\geq c\lVert \gamma(t)-\zeta\rVert\textup{ for }t\gg 1.$$
Consequently,
$$k_{\Omega\setminus S'}\big{(}\gamma (t_2),a(s_2)\big{)}\geq -\frac{1}{2}\log\frac{c\lVert \gamma (t_2)-\zeta\rVert}{\varepsilon}=t_2-\frac{1}{2}\log \frac{c\lVert\zeta-\eta\rVert}{\varepsilon}.$$
Notice that the above reasoning also works for the first compound of $(*),$ and so the proof is derived.
 \end{proof}
Let us observe that
\begin{equation}\label{nierownosc}
\inf_{z\in (p^u,\eta^u)}\,k_{\Omega\setminus S'}(\zeta^v,z)\geq \inf_{z\in (p^u,\eta^u)}\,k_{H_{\Omega}}(\zeta^v,z).
\end{equation}
\begin{lemma}\label{last lemma} For every $M>0$ there exist $v_0$ so that for every $v\geq v_0$
$$\inf_{z\in (p^u,\zeta^u)}\,k_{\Omega\setminus S'}(\zeta^{v},z)\geq M$$
for $u\gg v.$
\end{lemma}
\begin{proof}
By the localization of the Kobayashi metric:
$$\kappa_{\mathbb{B}(\zeta,s')\cap(\Omega\setminus S')}(z;X)\leq C\kappa_{\Omega\setminus S'}(z;X),\ z\in\mathbb{B}(\zeta,s)\cap\Omega\setminus S',\,X\in\mathbb{C}^n$$
for some $C>0$ (see Proposition \ref{lokalizacja}). Observe that the use of the localization is justified since $\Omega$ does not contain any complex line and so is $k$-hyperbolic.

Fix $u\gg v$ and $z\in (p^u,\zeta^u).$ Choose a smooth curve $\iota:[0,1]\rightarrow\Omega\setminus S'$ so that $\iota(0)=\zeta^{v},\,\iota(1)=z.$ Thus:
\begin{multline}\label{ostatnie}l_{k_{\Omega\setminus S'}}(\iota)\ \geq\ \begin{cases}
     l_{k_{H_\Omega}}(\iota)    & \quad \text{if }\ \iota\nsubseteq \mathbb{B}(\zeta,s)\cap (\Omega\setminus S')\\
   \frac{1}{C}\,l_{k_{\mathbb{B}(\zeta,s')\cap (\Omega\setminus S')}}(\iota)  & \quad \text{if }\ \iota\subset \mathbb{B}(\zeta,s)\cap (\Omega\setminus S')\\
  \end{cases} \\
  \textup{ } \\
\geq\ \begin{cases}
    -\frac{1}{2}\log\textup{dist}_{H_{\Omega}}(\zeta^{v})-\frac{1}{2}\log\textup{dist}_{H_{\Omega}}(\zeta^{u})+\log\varepsilon\,   & \quad \text{if }\ \iota\nsubseteq \mathbb{B}(\zeta,s)\cap (\Omega\setminus S')\\
 -\frac{1}{2C}\,\log\big{(}\log\,\lVert\zeta^v-\zeta\rVert-\log\,\lVert \zeta^u-\zeta\rVert\big{)}  & \quad \text{if }\ \iota\subset \mathbb{B}(\zeta,s)\cap (\Omega\setminus S')\\
  \end{cases} \\
\end{multline}
(in the last inequality we applied Lemma \ref{podstawowe}).
Since the last estimate does not depend on the choice of $\iota,$ letting $u\rightarrow\infty,$ we obtain the statement.
\end{proof}
Fix $M>0.$ Choose $u$ so big that $\eta^u\in\mathbb{B}(\eta,s).$ By (\ref{nierownosc}) there is $v_0$ so that
$$k_{\Omega\setminus S'}(\zeta^{v},(p^u,\eta^u))>M$$
for $v\geq v_0.$
On the other hand, by Lemma \ref{last lemma}, we have
$$k_{\Omega\setminus S'}(\zeta^v,(p^u,\zeta^u))>M.$$
Hence the triangle $[p^u,\zeta^u,\eta^u]$ is not $M$-thin, and so $\Omega\setminus S'$ is not hyperbolic. But this contradicts Proposition \ref{prop Gromov}.

\vspace{0.1cm}
\textbf{Case: $S'\nsubseteq S$}

The only thing that requires a comment in the previous case is to show that $[p,\eta)$ is an $(A,B)$-quasi-geodesic.
\begin{lemma}Assume $\Omega$ is a convex domain, $L\subset\Omega$ closed and $\mu\in\partial\Omega\setminus L$ is a strongly convex point of $\partial\Omega.$ Providing the segment $[z,\mu)\subset\Omega\setminus L,$ $[z,\mu)$ can be parametrized as an $(A,B)$-quasi-geodesic in $(\Omega\setminus L,k_{\Omega\setminus L}).$
\end{lemma}
\begin{proof}
Consider the standard parametrization $\alpha$ of $[z,\mu)$ (cf. Step II). We must show that
$$A^{-1}|s-t|-B\leq k_{\Omega\setminus L}\big{(}\alpha(t),\alpha(s)\big{)}\leq A|s-t|+B,\ \ s,\,t\geq 0.$$
Since $\Omega$ is convex, the estimate from below follows by considering an appropriate supporting hyperplane at $\eta.$ The estimate from below follows from (\ref{ostatnie}).
\end{proof}
The proof has been completed.
\end{proof}

\section{Proof of Theorem \ref{last thm}}
Without loss of generality we may assume $0\in\Omega.$ Put $$R=e^{-\inf \varphi},\ r=e^{-\sup\varphi}.$$
Clearly,
\begin{equation}\label{ostatn rozdzial}
\Omega_{r}=\Omega\times\mathbb{D}(r)_*\subset\Omega_\varphi\subset \Omega\times\mathbb{D}(R)_*=\Omega_{R}.
\end{equation}
Repeat the construction of the system of quasi-triangles from Theorem \ref{paris we} for $\Omega_r.$ Namely, choose points $0,\,\zeta,\,p$ so that
$$p=(0,\ldots,0,p_{n+1}),\ p_{n+1}\in (0,r),$$
and
$$\zeta\ \ \textup{ is so that }\ \textup{dist}_\Omega(0)=\lVert\zeta \rVert.$$
Next, define $p^T,\,\eta^T$ for $T>0.$ By Lemma \ref{product} $\Omega_r$ is weakly linearly convex. All estimates from above in Section 4 require only weak linear convexity. $\mathbb{C}$-convexity was essential therein for getting the lower estimates in Step 2. But, $\Omega_r,\,\Omega_R$ are in general not $\mathbb{C}$-convex, so we cannot deduce Theorem \ref{last thm} directly from Theorem \ref{paris we}. Hence, we must show why the constructed segments are quasi-geodesics in  $\Omega_r,\,\Omega_R.$

Every point $p\in\mathbb{C}^{n+1}$ we will consider every point as an element of $\mathbb{C}^n\times\mathbb{C},$ that is $p=(p',p_{n+1}).$ Let $\pi_{n+1}:\mathbb{C}^{n+1}\rightarrow\mathbb{C}^n$ be the projection on the first $n$ coordinates, i.e. $\pi_{n+1}:p=(p',p_{n+1})\mapsto p'.$

It suffices to focus only on the segment $[p,0)$ parametrized as in Step 2. We have:
\begin{multline*}
k_{\Omega_\varphi}(p^s,p^t)\geq k_{\Omega_r}(p^s,p^t)\geq k_{\Omega}\big{(}\pi_{n+1}(p^s),\pi_{n+1}(p^t)\big{)} \\
 \geq k_\Omega\big{(}\zeta'+e^{-2s}(p'-\zeta'),(\zeta'+e^{-2t}(p'-\zeta')\big{)}
 \geq\frac{1}{2}|s-t|.
\end{multline*}

Consequently, there are constants $A,\,B>0$ so that $[p,p^T],\,[p^T,\eta^T],\,[\eta^T,p]$ are $(A,B)-$quasi-geodesics for $T\gg 1.$ Observe that we may increase $A,\,B$ if necessary so as to ensure the constructed system forms also a system of $(A,B)$-quasi-geodesic triangles in $\Omega_R$ ($p^T,\,\zeta^T$ do not depend on $r$). Furthermore, by (\ref{ostatn rozdzial}) it is also a system of quasi-geodesic triangles in $\Omega_\varphi.$

We proceed to the main part of the proof. Assume the contrary, i.e., $(k_{\Omega_\varphi},\Omega_\varphi)$ is Gromov hyperbolic. By Proposition \ref{proposition weak linear} applied to $\Omega_R$ and the holomorphic contractibility, we can choose $T_0>0$ for which
$$k_{\Omega_\varphi}(\eta^{T_0},[p,0))\geq k_{\Omega_R}(\eta^{T_0},[p,0))>M.$$ Now Lemma \ref{podstawowe} implies that there is $T>T_0$ such that
$$k_{\Omega_\varphi}(\eta^{T_0},[p^T,\eta^T])\geq k_{\Omega_R}(\eta^{T_0},[p^T,\eta^T])>M.$$ Consequently, the triangle $[p,p^T],\,[p^T,\eta^T],\,[\eta^T,p]$ is not $M$-thin, and so $(\Omega_\varphi,k_{\Omega_\varphi})$ cannot be Gromov hyperbolic. This is a contradiction.

\subsection*{Acknowledgment} Part of this work was done during the stay  of the second named author in Singapore in May 2017 during {\it Complex Geometry, Dynamical Systems and Foliation Theory.} She would like to thank the Organizators for their hospitality and excellent working conditions. She also would like to thank John Erik Forn\ae ss for many helpfull remarks that essentialy improved the presentaion of the paper.

\end{document}